\documentclass[12pt, reqno]{amsart}

\usepackage{verbatim}
\usepackage{amssymb, amsmath, amsthm}
\usepackage{hyperref}
\usepackage{amscd}   % for commutative diagrams
\usepackage{fullpage}
\usepackage[all]{xy} % for complicated commutative diagrams
\usepackage{longtable}
\usepackage[normalem]{ulem}

\usepackage[usenames, dvipsnames]{color}
\usepackage{graphicx}
\newtheorem{lemma}{Lemma}[section]
\newtheorem{theorem}[lemma]{Theorem}
\newtheorem{proposition}[lemma]{Proposition}
\newtheorem{prop}[lemma]{Proposition}
\newtheorem{cor}[lemma]{Corollary}

\newtheorem*{theorem*}{Theorem}

\theoremstyle{definition}
\newtheorem{remark}[lemma]{Remark}

\newtheorem{example}[lemma]{Example}
\newtheorem{defn}[lemma]{Definition}

\newtheorem{claim*}{Claim}

\newcommand{\PP}{{\mathbb P}}
\newcommand{\CC}{{\mathbb C}}
\newcommand{\FF}{{\mathbb F}}
\newcommand{\QQ}{{\mathbb Q}}
\newcommand{\RR}{{\mathbb R}}
\newcommand{\ZZ}{{\mathbb Z}}
\newcommand{\NN}{{\mathbb N}}

\usepackage[mathscr]{euscript}

\DeclareMathOperator{\Aut}{Aut}

\DeclareMathOperator{\ord}{ord}

\DeclareMathOperator{\PGL}{PGL}

\numberwithin{equation}{section}
\numberwithin{table}{section}

\title{Dynamical Belyi maps}
\author{Jacqueline Anderson}
\author{Irene Bouw}
\author{Ozlem Ejder}
\author{Neslihan Girgin}
\author{Valentijn Karemaker}
\author{Michelle Manes}
\thanks{MM partially supported by NSF-HRD 1500481 (AWM ADVANCE  grant) and by the Simons Foundation grant \#359721.}

\email{jacqueline.anderson@bridgew.edu, irene.bouw@uni-ulm.de, ejder@usc.edu, \newline nslkr09@gmail.com, vkarem@math.upenn.edu, mmanes@math.hawaii.edu}

\begin{document}

%\irene{test comment}
%\val{test comment}
%\jackie{test comment}
%\michelle{test comment}
%\neslihan{test comment}
%\ozlem{test comment}

\begin{abstract}
We study the dynamical properties of a large class of rational maps
with exactly three ramification points. 
By constructing families of such maps, we obtain infinitely many conservative maps of fixed degree $d$; 
this answers a question of Silverman.
Rather precise results on the reduction
of these maps yield strong information on their $\QQ$-dynamics.

\end{abstract}

\maketitle
%\ifx
%    \thepage\undefined\def\jobname{ABEGKM_Covers}
%    \input{ABEGKM_Covers}
%\fi

%%%%%%%%%%%%%%%%%%%%%%%%%%%%%%%%%%%%%%%%%%%%%%%%%%%%%%%%%%%%%%%%%%%%%%%%%%%%%%%%
\section{Introduction}\label{sec:Introduction}

%%%%%%%%%%%%%%%%%%%%%%%%%%%%%%%%%%%%%%%%%%%%%%%%%%%%%%%%%%%%%%%%%%%%%%%%%%%%%%%%

Let $X$ be a smooth projective curve. A \textit{Belyi map} is a finite cover $f:X\to \PP^1$ that is branched exactly at $0, 1, \infty$.
In this paper we restrict to the special case that $X=\PP^1$, allowing
for iteration of $f$ and the study of dynamical behavior.  We write
$f^n$ for the $n$-fold composition of $f$ with itself.  If we further
insist that $f(\{ 0 , 1, \infty \}) \subseteq \{ 0, 1, \infty \}$,
then all iterates $f^n$ are also Belyi maps.  These maps have
been called \emph{dynamical Belyi maps}~\cite{SV},\cite{Zvonkin}.

Dynamical Belyi maps  are the central objects of study in the present work.
These maps are a special class of \emph{post-critically
finite} (PCF) maps; a map $f:\PP^1 \to \PP^1$ is PCF if each
ramification point has a finite forward orbit.
The study of PCF maps has a long history in complex and arithmetic dynamics, starting from Thurston's topological characterization of these maps in the early 1980s and continuing to the present day~\cite{BIJL},~\cite{BBLPPcensus},~\cite{LMY},~\cite{pilgrim}.

\subsection*{Outline and summary of results}
In \S\ref{sec:RET}, we describe a large class of  dynamical Belyi maps with exactly three ramification points: the genus-0 single-cycle normalized Belyi maps.
In addition, each of the three ramification points --- $0$, $1$, and $\infty$ --- is fixed.    
Maps $f: \PP^1 \to \PP^1$ for which every ramification point is fixed are called \emph{conservative rational maps} or \emph{critically fixed}, and have been the subject of some recent study~\cite{CGNPP},~\cite{Pconservative},~\cite{Tischler}.  

The genus-0 single-cycle Belyi maps are each naturally defined over $\QQ$ (Proposition \ref{prop:rigid}), a result which  suffices to answer in the negative a question of Silverman~\cite[p.\ 110]{JHSMod}: Is the number of $\PGL_2$-conjugacy classes of conservative rational maps of degree $d$ in $\QQ[z]$ or $\QQ(z)$  bounded independently of $d$?

In \S\ref{sec:GenFams}, we give explicit equations for two particular infinite families of maps in this class (Propositions~\ref{prop:GenFams} and~\ref{prop:GenFams4}),
 providing many new examples of maps of interest in complex and arithmetic dynamics.   

Studying the reduction of the maps from \S\ref{sec:RET} to
characteristic $p$ leads to one of our main results
(Theorem~\ref{thm:monored}): Given  a genus-0 single-cycle normalized Belyi
map $f$ of degree $d\geq 3$, we describe 
a necessary and sufficient condition for
the reduction of $f$ modulo a prime $p$   to be the monomial map
$x^d$. 

In \S\ref{sec:Dynamics}, we use this result, together with standard
local-global theorems in arithmetic dynamics, to study the dynamical
behavior of the genus-0 single-cycle Belyi maps $f$. Our main result
(Theorem~\ref{qdynamicsGeneral}) gives conditions for the set of
$\QQ$-rational preperiodic points to coincide with the set of
$\QQ$-rational fixed points of $f$. For one family of explicit Belyi
maps we determine the set of preperiodic points exactly
(Proposition \ref{fixedptlemma}).

\subsection*{Acknowledgments}
This project began at the Women in Numbers Europe 2 conference at the
Lorentz Center. We thank the Lorentz Center for providing excellent working conditions, and we thank the Association for Women in Mathematics for supporting WIN-E2 and other research collaboration conferences for women through their NSF ADVANCE grant.  We also thank the referee for numerous helpful comments, all of which greatly improved the paper.

%\ifx
%    \thepage\undefined\def\jobname{ABEGKM_Covers}
%    \input{ABEGKM_Covers}
%\fi

%%%%%%%%%%%%%%%%%%%%%%%%%%%%%%%%%%%%%%%%%%%%%%%%%%%%%%%%%%%%%%%%%%%%%%%%%%%%%%%%
\section{Background on Belyi maps}\label{sec:RET}

%\section{Riemann existence theorem}
%%%%%%%%%%%%%%%%%%%%%%%%%%%%%%%%%%%%%%%%%%%%%%%%%%%%%%%%%%%%%%%%%%%%%%%%%%%%%%%%

 A \textit{Belyi map} is a finite cover $f: X \to \mathbb{P}^1$
 of smooth projective curves defined over $\CC$ that is branched
 exactly at $0,1,\infty$.  Belyi maps can be described topologically
 as \emph{dessins d'enfants}~\cite{Sch} or combinatorially in terms of
 generating systems.  

\begin{defn}\label{gensysdef}
Fix an integer $d>1$. A \textit{generating system} of degree $d$ is a
triple ${\boldsymbol \rho}=(\rho_1, \rho_2,\rho_3)$ of
permutations $\rho_i\in S_d$ that satisfy
\begin{itemize}
\item $\rho_1\rho_2\rho_3=1$,
\item $G:=\langle \rho_1, \rho_2,\rho_3\rangle\subset S_d$ acts
  transitively on $\{1, 2, \ldots, d\}$.  
\end{itemize}

The \textit{combinatorial type} of ${\boldsymbol \rho}$ is a tuple
$\underline{C}:=(d; C(\rho_1), C(\rho_2),C(\rho_3))$, where $d$ is
the degree and $C(\rho_i)$ is the conjugacy class of $\rho_i$ in
$S_d$.
\end{defn}

Two generating systems $\boldsymbol \rho$ and $\boldsymbol \rho'$
are \textit{equivalent} if there exists a permutation $\tau\in S_d$ such
that $\rho_i'=\tau\rho_i\tau^{-1}$ for $i=1,2,3$. 
Furthermore, two Belyi maps $f_i : X_i \to \PP^1$ are
 \textit{isomorphic} if there exists an isomorphism $\iota : X_1 \to
 X_2$ making
\[
\xymatrix{
X_1 \ar[dr]^{f_1} \ar[rr]^{\iota} && X_2\ar[dl]_{f_2} \\
 & \PP^1  \\
}
\]
commutative.  In particular, the $f_i$ have the same branch locus.
Riemann's Existence Theorem (\cite[Theorem 2.13]{volklein}) yields a
bijection between equivalence classes of generating systems and
isomorphism classes of Belyi maps $f:X\to \PP^1$.

Let ${\boldsymbol \rho}$ be a generating system.  The conjugacy class
$C_i:=C(\rho_i)$ of $S_d$ corresponds to a partition $\sum_{j=1}^{r_i}
n_j=d$ of $d$. The \textit{length} $r_i=r(C_i)$ of $C_i$ is the number
of cycles of the elements of $C_i$. Note that we include the
$1$-cycles. The nonnegative integer $g:=(d+2-r_1-r_2 -r_3)/2$ is called the \textit{genus} of the generating system.  If $f:X\to \mathbb{P}^1$
is the Belyi map corresponding to $\boldsymbol{\rho}$ then $g=g(X)$ and $r_i=|f^{-1}(t_i)|$ is the cardinality of the inverse image of the
$i$th branch point $t_i$.

 In this paper we only consider the case that $g=g(X)=0$.  We write
\[
f: \mathbb{P}^1_x \to \mathbb{P}^1_t, \qquad x\mapsto t=f(x);
\]
the subscript indicates the coordinate of the corresponding projective
line. Note that we write $f$ both for the rational function
$f(x)\in \CC(x)$ and the cover defined by it.

 We restrict, moreover, to the case that $C_i$ is the conjugacy class
 of a single cycle, i.e.~for each $i\in \{1,2,3\}$ the partition
 $\sum_{j}n_j$ corresponding to $C_i$ contains a unique part $n_j$
 different from $1$. We denote this part by $e_i$. Formulated
 differently, the corresponding Belyi map $f$ has a unique
 ramification point above $t_i$, with ramification index $e_i$. The
 assumption that $g=g(X)=0$ translates to the condition
\[
2d+1=e_1+e_2+e_3.
\]
We call this situation the \textit{genus-$0$ single-cycle case}, and
we write $(d; e_1, e_2, e_3)$ for the combinatorial type of $f$. More
generally, given four integers $d, e_1, e_2, e_3$ satisfying $2\leq
e_i \leq d$ and $2d+1=e_1+e_2+e_3$, we call $(d; e_1, e_2, e_3)$ an
\textit{abstract combinatorial type of genus $0$}.

We say that a genus-$0$ single-cycle Belyi map $f$ is
\textit{normalized} if its ramification points are $0, 1, \text{ and
} \infty$, and if, moreover,
\[
f(0)=0, \quad f(1)=1, \quad f(\infty)=\infty.
\]
 Since $f$ has three ramification points, every isomorphism class of
 covers contains a unique normalized Belyi map. 

\begin{remark}
In this paper we always assume that $f$ has exactly three (as opposed
to at most three) ramification points.  We may allow some $e_i = 1$
without substantial change to our theoretical results. In this case
the Belyi map $f$ is Galois, and its Galois group is cyclic. Therefore $f$
is linearly conjugate (in the sense of Definition~\ref{def:linconj}
below) to the pure power map $f(x) = x^d$.  This case is well
understood, so we omit the details.

For instance, if we include the case $e_2=1$ (i.e.~the combinatorial type $(d; d, 1, d)$) in our considerations, this only affects a few results: the count of normalized Belyi maps of each degree in Corollary~\ref{cor:BelyiNum2} slightly changes, and the statement of Proposition~\ref{prop:dynamicsp=2}
needs to be altered by including the assumption
$\nu>0$ for the conclusion to hold. 
If one of the other $e_i$ is trivial, slightly different but equally straightforward adaptions are necessary.
\end{remark}

The following proposition states that the triple of conjugacy classes
corresponding to the combinatorial type $(d; e_1, e_2, e_3)$ is rigid
and rational.

\begin{prop}\label{prop:rigid}
Let $\underline{C}:=(d; e_1, e_2, e_3)$ be an abstract combinatorial type of
genus $0$. Then there exists a unique normalized Belyi map
$f:\PP^1_\mathbb{C}\to \mathbb{P}^1_\mathbb{C}$ of combinatorial
type $\underline{C}$. Moreover, the rational map $f$ may be defined
over $\mathbb{Q}$.
\end{prop}

\begin{proof}
We have already seen that the number of normalized Belyi maps of
combinatorial type $\underline{C}$ is the cardinality of the finite set 
\[
\left\{ (\rho_1, \rho_2, \rho_3) \in C_1 \times C_2 \times C_3 \text{
  such that } \langle \rho_i \rangle \subset S_d \text{ transitive, }
\prod_i \rho_i = e \right\} \Bigg/ \mathord{\sim}
\]
of generating systems of combinatorial type $\underline{C}$ up to
equivalence.   Liu and Osserman prove in
\cite[Lemma 2.1]{liuosserman} that in the genus-$0$ single-cycle case
the cardinality of this set is $1$.  

We have already seen that the (unique) isomorphism class of Belyi maps
of the given type contains a unique normalized one. This implies
that the coefficients of this unique map are in $\QQ$.
Alternatively, one may also apply the rigidity criterion from
\cite[Corollary 3.13]{volklein}.
\end{proof}

\begin{remark}
Let $f:\PP^1_x\to \PP^1_t$ be a normalized Belyi map of combinatorial
type $\underline{C}=(d; e_1, e_2, e_3)$, i.e.~a genus-$0$ single-cycle
map. The coordinates $x$ and $t$ are completely determined by the
normalization of the ramification and branch points.  Proposition
\ref{prop:rigid} states that the (unique) rational function defining the
Belyi map $f$ satisfies $f(x)\in \QQ(x)$.
\end{remark}

Normalized Belyi maps are examples of \textit{conservative rational
maps}, i.e.~rational maps $f: \PP^1 \to \PP^1$ such that every
ramification point is fixed.  In~\cite[top of p.\ 110]{JHSMod},
Silverman asked if the number of $\PGL_2$-equivalence classes of
conservative maps $\PP^1 \to \PP^1$ of degree $d$ in $\QQ[z]$ or
$\QQ(z)$ may be bounded independently of $d$.  We use
Proposition~\ref{prop:rigid} to answer this question.

\begin{defn}\label{def:linconj}
Two rational functions $f, g: \PP^1_K \to \PP^1_K$ are
\textit{linearly conjugate over a field} $K$ if there is a $\phi \in
\Aut(\PP^1_K) \cong \PGL_2(K)$ such that $f^\phi:=\phi^{-1} f \phi =
g$.
\end{defn}

Note that linear conjugacy respects iteration; that is, if $f^\phi =
g$, then $\left(f^n\right)^\phi = \left(f^\phi\right)^n=g^n$. 
Note also that for normalized genus-$0$ single-cycle Belyi maps
Proposition \ref{prop:rigid} implies that linear conjugacy over $\CC$
is the same as linear conjugacy over $\QQ$, so we may omit the
mention of the field.

\begin{remark}
Definition~\ref{def:linconj} is the standard equivalence relation used
in arithmetic and holomorphic dynamics.  It is different from the
notion of isomorphic covers introduced at the beginning of this
section in which the target $\PP^1$ is fixed.  When we say that maps
are conjugate, we mean linearly conjugate in the sense of this
definition, where the same isomorphism is used on both the source and
the target $\PP^1$.
\end{remark}

\begin{lemma}\label{lem:BelyiNum}
Let $f: \PP^1_\CC \to \PP^1_\CC $ be a normalized Belyi map of
combinatorial type $(d; e_1, e_2, e_3)$ and let $g$ be a normalized
Belyi map with combinatorial type $(d; e'_1, e'_2, e'_3)$.  Then $f$
and $g$ are linearly conjugate over $\CC$ if and only if there is some
permutation $\sigma \in \mathcal S_3$ such that $e_i = e'_{\sigma(i)}$
for $i = 1, 2, 3$.
\end{lemma}

\begin{proof}
The ramification index of a point is a dynamical invariant in the
following sense: for a nonconstant function $f: \PP^1 \to \PP^1$, any
point $\alpha \in\PP^1$, and any $\phi \in \PGL_2$, the ramification
indices satisfy $e_\alpha(f) = e_{\phi^{-1}(\alpha)}(f^\phi)$.  

In the case of normalized Belyi maps, the only points of $\PP^1$ with
ramification index greater than one are $t_1=0$, $t_2=1$, and
$t_3 = \infty$. Assume that $f$ and $g$ are normalized Belyi maps with
$f^\phi = g$ for some $\phi\in \PGL_2$. Then we may define a
permutation $\sigma\in S_3$ by
\[
e_i= \text{ ramification index of $f$ at $t_i$ } = \text{ ramification
  index of $g$ at $\phi^{-1}(t_i)$ } = e'_{\sigma(i)}.
\]

Conversely, given a permutation $\sigma\in S_3$ and a normalized Belyi
map $g$ of combinatorial type $(d; e_{\sigma(1)}, e_{\sigma(2)},
e_{\sigma(3)})$, there exists a unique $\phi\in \PGL_2$ satisfying
$\phi(t_i) = t_{\sigma(i)}$ for $i=1,2,3$. Proposition
\ref{prop:rigid} implies that $g=f^\phi$, since both normalized Belyi
maps have the same combinatorial type.
\end{proof}

Lemma~\ref{lem:BelyiNum}, together with the rationality
result from Proposition~\ref{prop:rigid}, answers Silverman's question in the
negative.  

\begin{cor}\label{cor:BelyiNum2}
The number of $\PGL_2$-conjugacy classes of conservative polynomials in $\QQ[z]$ of degree~$d\geq 3$ is at least $\left\lfloor \frac{d-1}2 \right\rfloor$.
The number of $\PGL_2$-conjugacy classes of nonpolynomial conservative rational maps in $\QQ(z)$ of degree~$d\geq 4$ is at least 
\begin{equation}\label{corsum}
 \sum_{i=1}^{\left\lfloor{\frac{d-1}3}\right\rfloor} \left\lfloor{\frac{d+1-3i}{2}}\right\rfloor.
 \end{equation}
\end{cor}

\begin{proof}
We count the number of $\PGL_2$-conjugacy classes of genus-0
single-cycle normalized Belyi maps in degree $d$, which serves as a
lower bound for all conservative rational maps in the given degree, up
to linear conjugacy.  By Lemma~\ref{lem:BelyiNum}, this equals the
number of partitions of $2d+1$ into exactly three parts such that each
part is at least $2$ and none exceed $d$.  The number of partitions
equals the cardinality of
 \begin{equation}\label{eq:count}
 \{2\leq e_1 \leq e_2 \leq e_3 \leq d \quad \mid \quad e_1 + e_2 + e_3 =2d+1\}.
 \end{equation}

If $f$ is a polynomial of degree $d$, then the ramification index $e_3
= e_\infty(f) = d$.  Hence, it is enough to count pairs $(e_1,e_2)$
such that $2 \leq e_1 \leq e_2 \leq d-1$ and $e_1+ e_2 = d+1$.  We see
that $e_1$ can take on $\left\lfloor \frac{d-1}2 \right\rfloor$
distinct values, and $e_2$ is determined by $e_1$.

To count nonpolynomial maps, we use the same strategy for every possible value of $e_3\leq d-1$.  Fixing $e_3 = d-i$, we count pairs $(e_1, e_2)$ such that 
\[
2 \leq e_1 \leq e_2 \leq d-i \quad \text{ and } \quad e_1 + e_2  = d+i+1.
\]  
These constraints give that $2i+1 \leq e_1 \leq \lfloor \frac{d+i+1}2\rfloor$, yielding $\lfloor{\frac{d+1-3i}{2}}\rfloor$ distinct possibilities.

Finally, the constraints in~\eqref{eq:count} require that $d-1 \geq e_3 \geq \left\lceil \frac{2d+1}3 \right\rceil$.  Writing $e_3 = d-i$ gives $1 \leq i \leq \lfloor \frac {d-1} 3 \rfloor$, and the result follows.
\end{proof}

\begin{remark}\mbox{} \begin{enumerate}
\item The sum in \eqref{corsum} %Corollary~\ref{cor:BelyiNum2}
can be calculated explicitly. Let  $N(d)$ be the number of $\PGL_2$-conjugacy classes of genus-0 single-cycle normalized Belyi maps of degree $d$ (including the polynomial maps). Then

\begin{equation*}\label{eq:N(d)}
N(d) = \frac{1}{12}(d^2+4d-c), \quad \text{where } c = \begin{cases}
5 & \text{ if } d \equiv 1 \pmod 6\\
8 & \text{ if } d \equiv 4 \pmod 6\\
9 & \text{ if } d \equiv 3, 5 \pmod 6\\
12 & \text{ if } d \equiv 0, 2 \pmod 6\\
\end{cases}
\end{equation*}
This serves as a lower bound for the number of conservative rational maps of degree $d$ in $\QQ(z)$.

\item A result of Tischler \cite{Tischler} counts  
all monic conservative polynomials $f$ in $\CC[z]$ normalized by
requiring that $f(0)=0$. Tischler shows that there are exactly
$\binom{2d-2}{d-1}$ such maps. This result is generalized
in \cite{CGNPP}, where conservative rational maps rather than
polynomials are counted. Note that in this more general situation
counting is more difficult as one has to impose the condition that $f$ is
conservative. In our setting this condition is automatically satisfied.
\end{enumerate}
\end{remark}

%\ifx
%    \thepage\undefined\def\jobname{ABEGKM_Covers}
%    \input{ABEGKM_Covers}
%\fi

%%%%%%%%%%%%%%%%%%%%%%%%%%%%%%%%%%%%%%%%%%%%%%%%%%%%%%%%%%%%%%%%%%%%%%%%%%%%%%%%
\section{Families of Dynamical Belyi Maps}\label{sec:GenFams}

%%%%%%%%%%%%%%%%%%%%%%%%%%%%%%%%%%%%%%%%%%%%
In this section we determine some families of normalized dynamical Belyi maps explicitly. These results yield infinitely many explicit maps to which we can apply the dynamical system results from Section \ref{sec:Dynamics}.

\begin{proposition}\label{prop:GenFams} 
     If a normalized Belyi map $f$ has combinatorial type $(d; d-k,k+1,d)$, then $f(x)$ is given by
            \[f(x)=cx^{d-k} (a_0x^k+\ldots +a_{k-1}x+a_k), \]
    where \[a_i:=  \frac{(-1)^{k-i}}{(d-i)} \binom{k}{i} \hspace{2mm} \text{and} \hspace{2mm} c=\frac{1}{k!}              \prod_{ j=0}^{k}(d-j). \]

\end{proposition}
  \begin{proof} It is clear that the ramification index $e_3$ is $d$,
     since $f$ is a polynomial, and that $e_1$ is $d-k$.  We need to
     show that the ramification index of $e_2$ is $k+1$. The
     derivative of $f$ is given by 
\begin{equation*}
\begin{split}
f'(x)&=
     c\sum_{i=0}^{k}\frac{(-1)^{k-i}}{(d-i)}\binom{k}{i}(d-i)x^{d-i-1} \\
     &= cx^{d-k-1}\sum_{i=0}^{k} (-1)^{k-i} \binom{k}{i}x^{k-i} \\
     &=(-1)^kcx^{d-k-1}(x-1)^k. \\ 
\end{split} 
\end{equation*}
Hence the only ramification points of $f$ are $0,1$ and $\infty$ and
     the ramification index $e_2$ is equal to $k+1$.  We are left to
     show that $f(1)=1$ which is equivalent to showing that
\[ 
\sum_{i=0}^{k}{\left(a_i \prod_{
    j=0}^{k}(d-j) \right)}
     = \sum_{i=0}^{k}{\left((-1)^{k-i}\binom{k}{i}\prod_{ \substack{j=0\\
     j\neq i}}^{k}(d-j) \right)}=k!. \] 
We  first show that in the
     above sum, the coefficients of $d^l$ for each $1 \leq l \leq k$
     are $0$ and the constant term is $k!$.  Notice that the
     coefficient of $d^k$ is $\sum_{i=0}^k{(-1)^{k-i}\binom{k}{i}}$,
     which is $0$ since this is the binomial expansion of $(x-1)^k$
     evaluated at $1$.  Similarly, for the other terms $d^l$ for $
     l\geq 1$, we obtain a sum
     $\sum_{i=0}^k{(-1)^{k-i}\binom{k}{i}}p(i)$ where $p(x)$ is a
     polynomial of degree less than $k$. This sum is also zero
     by \cite[Corollary~2]{Ruiz}.  The constant coefficient
     is \[ \sum_{i=0}^{k}{\left((-1)^{i}\binom{k}{i}\prod_{ \substack{j=0\\
     j\neq i}}^{k}(j) \right)}=\binom{k}{0}k!=k! . \]
              
 \end{proof}

\begin{remark}\label{rem:altGenFams}
We now provide an alternative proof of Proposition \ref{prop:GenFams}. A variant of this proof can be found in the unpublished master's thesis of Michael Eskin; his PhD-thesis~\cite[Proposition 5.1.2]{eskin} contains a slightly weaker version of the result. 

\noindent To have the correct ramification at $0$, $1$, and $\infty$, we see that $f$ must be of the form 
\begin{equation}\label{formf}
f(x) = x^{d-k}f_1(x)
\end{equation}
for some $f_1(x) = \sum_{i=0}^k c_i (x-1)^i$, such that 
\[
f'(x) = (-1)^k cx^{d-k-1}(x-1)^k
\]
for some $c \in \mathbb{C}^{\times}$.
This implies  that 
\[
c(x-1)^k = (d-k)f_1 + xf'_1 = c_k d (x-1)^k + \sum_{i=0}^{k-1} \left((d-k+i)c_i + (i+1)c_{i+1}\right)(x-1)^i.
\]
This yields a recursive formula for the $c_i$, from which it follows that 
\[
c_i =(-1)^i \binom{d-k+i-1}{i}
\]
for all $i=0,\ldots,k$, and
\[
c = \binom{d-1}{k} d.
\] 
Substituting these values for $c_i$ and $c$ back into Equation \eqref{formf}, the reader may check we obtain the claimed result.
\end{remark} 
 
\begin{example}\label{exa:GenFams}
The unique normalized Belyi map $f$ of combinatorial type $(d; d-1, 2, d)$ is given by
\[
f(x)=-(d-1)x^d+dx^{d-1}.
\]
\end{example}

\begin{proposition}\label{prop:GenFams4} 
The unique normalized Belyi map $f$ of combinatorial type $(d;
 d-k,2k+1,d-k)$ is given by
\[
            f(x)=x^{d-k}\left( \frac{a_0x^k-a_1x^{k-1}+\ldots
            +(-1)^ka_k}{(-1)^k a_k x^k +\ldots -a_1x+ a_0}\right),
 \]
            where 
\[ 
a_i:=\binom{k}{i}\prod_{k+i+1\leq j \leq 2k} (d-j)\prod_{0\leq j \leq i-1}(d-j) =
k!\binom{d}{i}\binom{d-k-i-1}{k-i}.
\] 
\end{proposition}

\begin{proof}
The combinatorial type $(d; d-k,2k+1,d-k)$
is characterized by the fact that the ramification at $x=0, \infty$ is
given by the same conjugacy class in the sense of
Definition \ref{gensysdef}. This implies that the Belyi map $f$ admits
an automorphism in the following sense: Write $\psi(x)=1/x$.   Then
$f^\psi = \psi^{-1}\circ f\circ \psi$ has the same combinatorial type as $f$. 
% (Note that $\psi^{-1}(x) = 1/x$ also.)
 By Proposition \ref{prop:rigid}, $f$ is the
unique normalized Belyi map of the given type, so $\psi^{-1}\circ f\circ \psi = f$. 

From this, it immediately follows that
we may write
\[
f=x^{d-k}f_1/f_2,\qquad \text{ with } \qquad f_2(x)=x^{k}f_1(1/x).
\] 

Let $f=g/h$ with 
    \[
      g(x)=x^{d-k} \sum_{i=0}^{k}{(-1)^{k-i}a_{k-i}x^{i}} \quad \text{and} \quad
      h(x)=\sum_{j=0}^{k}{(-1)^j a_{j}x^j}, 
\] 
where the $a_i$ are as in the statement of the proposition.  It is
clear that the ramification at $x=0, \infty$ is as required. Moreover,
we see that 0, 1, and $\infty$ are fixed points of $f$. 

It therefore remains to determine the
ramification at $x=1$. More precisely, we need to show that the
derivative satisfies
\begin{equation}\label{eq:gf1}
f'(x)=\frac{cx^{d-k-1}(x-1)^{2k}}{h(x)^2},
\end{equation}
for some nonvanishing constant $c$. 
Write $g'h-gh'=x^{d-k-1}\sum_l c_lx^l$. 
We have
\[
c_l=(-1)^{k+l}\sum_{j=0}^l(d-k+l-2j)a_{k-l+j}a_j.
\]
Here we have used the convention that $a_i = 0$ if $i>k$ or $i<0$.
Equation \eqref{eq:gf1} on the $a_i$ therefore translates to
\begin{equation}\label{eq:gf2}
\begin{split}
c&=c_{2k}=(-1)^k(d-k)a_0a_k,\\
c_l&=(-1)^l c \binom{2k}{l}=(-1)^{k+l}\binom{2k}{l}(d-k)a_0a_k, \qquad l=0, \ldots, 2k.
\end{split}
\end{equation}

Hence, to prove \eqref{eq:gf2} it suffices  to prove the following:
    
\begin{equation} \label{l-coeff}  
\sum_{j=0}^{l} {(d-k+l-2j) a_{k-l+j}a_j}=\binom{2k}{l} (d-k)a_ka_0 
\end{equation} 
for every $l \leq k$.
(In fact, $l$  runs from $0$ to $2k$, but by symmetry it suffices to look at $l\leq k$.) 

Here and for the rest of the proof, we use the convention that $\binom{n}{m}=0$ if $m\leq 0$ or $m\geq n$.
Hence, the right hand side of \eqref{l-coeff} translates to
  
     \begin{equation*}
     \begin{split}
      \binom{2k}{l}(d-k)a_ka_0 &= \binom{2k}{l} (d-k)k!\binom{d}{k}k!\binom{d-k-1}{k} \\
                                                  &= d(k!)^2\binom{2k}{l}\binom{d-1}{2k}\binom{2k}{k}.\\
       \end{split}
    \end{equation*}
   We write $d-k+l-2j$ as the difference of $d-k+l-j$ and $j$.  Then the left hand side of \eqref{l-coeff} becomes    
    
     \begin{equation*} 
     d(k!)^2 \sum_{j=0}^l \left(\binom{d}{j}\binom{d-1}{k-l+j}-\binom{d-1}{j-1}\binom{d}{k-l+j}                   \right) \binom{d-2k+l-j-1}{l-j}\binom{d-k-j-1}{k-j}.
     \end{equation*}

    Hence, dividing both sides of (\ref{l-coeff}) by $d(k!)^2$, we find that we need to prove that the following equation holds for all integers $d,k$ and $l$ such that $d\geq 2k+1$ and $l \leq k$:   
    \begin{equation} 
    \begin{split}  \label{last identity} \sum_{j=0}^l \left(\binom{d}{j}\binom{d-1}{k-l+j}-\binom{d-1}     {j-1}\binom{d}{k-l+j}        \right)& \binom{d-2k+l-j-1}{l-j}\binom{d-k-j-1}{k-j}   \\
   &= \binom{2k}{l}\binom{d-1}{2k}\binom{2k}{k}.
   \end{split} 
   \end{equation}  
We fix $k,l$ such that $l\leq k$ and define $F_{k,l}(d,j)$ as the quotient of the $j$th term in the sum on the left hand side of Equation \eqref{last identity} by $\binom{2k}{k}\binom{d-1}{2k}\binom{2k}{l}$. 
Note that $F_{k,l}(d,j) = 0$ when $j > l$; this allows us to restate Equation \eqref{last identity} in the following form:
\begin{equation}\label{Fsum} 
\sum_{j=0}^{\infty} F_{k,l}(d,j)=1. 
\end{equation}

To prove that Equation \eqref{last identity}, or equivalently \eqref{Fsum}, holds for every value of $d \geq 2k+1$, we first prove it for $d=2k+1$, and then show that $ \sum_{j=0}^{\infty}F_{k,l}(d+1,j)=\sum_{j=0}^{\infty}F_{k,l}(d,j)$ for any $d\geq 2k+1$. 
   
So first suppose that $d=2k+1$. 
Then Equation \eqref{last identity} holds, since 
\begin{equation*}
\begin{split}
 \sum_{j=0}^l \left(\binom{2k+1}{j}\binom{2k}{k-l+j}- \binom{2k}{j-1}   \binom{2k+1}{k-l+j} \right)&\\ = \sum_{j=0}^l \binom{2k}{j}\binom{2k}{k-l+j}- \sum_{j=0}^{l-1} \binom{2k}{j} \binom{2k}{k-l+j}&= \binom{2k}{l}\binom{2k}{k}. 
\end{split}
\end{equation*}
    
Next, to show that $ \sum_{j=0}^{\infty}F_{k,l}(d+1,j)=\sum_{j=0}^{\infty}F_{k,l}(d,j)$, we write 
\[
 \sum_{j=0}^{\infty}{\left(F_{k,l}(d+1,j)-F_{k,l}(d,j)\right)}
 \]
  as a telescoping series and use an algorithm due to Zeilberger \cite{Zeilberger} that proves hypergeometric identities involving infinite sums of binomial coefficients. More explicitly, running this algorithm in Maple produces an explicit function $G_{k,l}(d,j)$ which satisfies
  \[
  F_{k,l}(d+1,j)-F_{k,l}(d,j)=G_{k,l}(d,j+1)-G_{k,l}(d,j).
  \]
One may check that $G_{k,l}(d,0)=0$. Moreover, $G_{k,l}(d,j)=0$ for all $j>l$ since the same is true for $F_{k,l}(d,j)$.
Hence,  $\sum_{j=0}^{\infty}(F_{k,l}(d+1,j)-F_{k,l}(d,j))$ equals
\[
\sum_{j=0}^{\infty} ({G_{k,l}(d,j+1)-G_{k,l}(d,j)})=G_{k,l}(d,l+1)-G_{k,l}(d,0)=0. \qedhere
\] 
      
\end{proof}

\begin{example}\label{prop:GenFams2} 
If a normalized Belyi map $f$ has combinatorial type $(d; d-1,3,d-1)$, then $f(x)$ is given by
\[
f(x)=x^{d-1} \frac{(d-2)x - d}{-dx + (d-2) }.
\]
(Note that necessarily $d \geq 3$ in this case.)
\end{example}

\begin{example}\label{prop:GenFams3}
If a normalized Belyi map $f$ has combinatorial type $(d; d-2,5,d-2)$, then $f(x)$ is given by
\[
f(x)=x^{d-2} \left(\frac{(d-3)(d-4)x^2 - 2d(d-4) x+ d(d-1)}{d(d-1)x^2 -2d(d-4)x+(d-3)(d-4)}\right).
\]
(Note that necessarily $d \geq 5$ in this case.)
\end{example}

%\ifx
%    \thepage\undefined\def\jobname{ABEGKM_Covers} \input{ABEGKM_Covers}
%\fi

%%%%%%%%%%%%%%%%%%%%%%%%%%%%%%%%%%%%%%%%%%%%%%%%%%%%%%%%%%%%%%%%%%%%%%%%%%%%%%%%
\section{Reduction properties of normalized Belyi maps}\label{sec:BadRed}

Let 
\[
f:\PP^1_x \to \PP^1_t ,\qquad x\mapsto f(x):=t
\]
 be a normalized Belyi map of combinatorial type $\underline{C}:=(d; e_1,
 e_2, e_3)$. Proposition \ref{prop:rigid} implies that
 $f(x)\in \QQ(x)$ is a rational function with coefficients in $\QQ$.
 We start by defining the reduction of $f$ at a rational prime
 $p$. 
Since we assume that the rational function $f$ is normalized, we may
 write
\[
f(x)=\frac{f_1(x)}{f_2(x)},
\]
where $f_1, f_2\in \QQ[x]$ are polynomials that are relatively
prime. Multiplying numerator and denominator by a common constant $c \in \QQ_{>0}$, we
may assume that $f_1, f_2\in \ZZ[x]$. 

For $k=1,2$, write
 \begin{equation}\label{eqn:deffi}
 f_k=c_k\tilde{f}_k, \quad \text{with $\tilde{f}_k\in \ZZ[x]$ a polynomial
of content $1$}
\end{equation}
Let $c=c_1/c_2$. The assumption that $f(1)=1$ translates to
\[
\tilde{f}_2(1)=c\tilde{f}_1(1)\in \ZZ.
\]
Note that $\tilde{f}(x):=\frac{\tilde{f}_1(x)}{\tilde{f}_2(x)}$ need
no longer be normalized.  The ramification points of $\tilde{f}$ are
still $0, 1, \infty$ but $\tilde{f}(1)$ need not be $1$. 
Nonetheless, it
makes sense to consider the reduction of $\tilde{f}$ modulo $p$. We
denote the reduction  of $\tilde{f}_k$ by $\overline{f}_k$ 
and put
\[
\overline{f}=\frac{\overline{f}_1}{\overline{f}_2}, \qquad \overline{f}_k
\in \FF_p[x].
\]
The definition of $\tilde{f}$ implies that $\overline{f}\neq 0$. 
We claim that in our situation $\overline{f}$ is not a constant. The proof below is inspired by a remark in \cite[Section 4]{ossermanrational}; note however that Osserman works only with maps in characteristic $p$, while we consider reduction to characteristic $p$ of maps in characteristic zero.

\begin{prop}\label{prop:rednonconstant}
Let $f$ be a normalized Belyi map of combinatorial type $\underline{C}:=(d; e_1,
 e_2, e_3)$. 
\begin{itemize}
\item[(1)] The reduction $\overline{f}$ is nonconstant.
\item[(2)] We have $\overline{f}(0)=0$ and 
$\overline{f}(\infty)=\infty$.
\item[(3)] We have $\overline{f}(1) \neq 0, \infty$. 
\end{itemize}
\end{prop}

\begin{proof}
Define $\tilde{f}_1$ and $\tilde{f}_2$ as in Equation~\eqref{eqn:deffi}. Let $i$
(resp.~$j$) be maximal such that the coefficient of $x^i$ in $\tilde{f}_1$
(resp.~of $x^j$ in $\tilde{f}_2$) is a $p$-adic unit.  
Note that the reduction $\overline{f}$ of $f$ is constant if and only if
%$\tilde{f}_1\equiv \tilde{f}_2\pmod{p}$. \jackie{Should this be $\tilde{f}_1\equiv a\tilde{f}_2\pmod{p}$ for some constant $a$?} 
$\tilde{f}_1\equiv a\tilde{f}_2\pmod{p}$ or $a\tilde{f}_1\equiv
\tilde{f}_2\pmod{p}$ for some constant $a$.  It follows that if
$\overline{f}$ is constant, then $i=j$.

The definition of the combinatorial type implies that
\[
e_1\leq i\leq d=\deg(\tilde{f}_1), \qquad 0\leq j\leq d-e_3=\deg(\tilde{f}_2).
\]
Since $e_2$ is a ramification index, we have that 
\[
e_2=2d+1-(e_1+e_3)\leq d.
\]
 This implies that
$e_1+e_3-d>  0$. It follows that
\begin{equation}\label{eq:rednonconstant}
i\geq e_1>d-e_3\geq j.
\end{equation}
This implies that $\overline{f}$ is nonconstant, and (1) is proved.

Equation \eqref{eq:rednonconstant} also implies that
\[
\deg(\overline{f}_1)=i>j=\deg(\overline{f}_2).
\]
This implies that $\overline{f}(\infty)=\infty$. 

Applying the same argument  to the minimal $i'$ (resp.\ $j'$) such that
the coefficient of $x^{i'}$ in $\tilde{f}_1$ (resp.\ the coefficient of
$x^{j'}$ in $\tilde{f}_2$) is a $p$-adic unit shows that 
\[
\ord_0(\overline{f}_1)=i'>j'=\ord_0(\overline{f}_2).
\]
We conclude that $\overline{f}(0)=0$, thus proving (2).

It remains to show that $\mu:=\overline{f}(1)\neq 0, \infty$.  We  have 
\[
\ord_0(\overline{f}_1)\geq e_1, \qquad \ord_1(\overline{f}_1)\geq e_2.
\]
We assume that $\mu=0$, i.e.~$\overline{f}(0)=\overline{f}(1)=0$. This
implies that
\[
e_1+e_2\leq \deg(\overline{f}_1)\leq \deg(f_1)=d.
\]
This yields a contradiction with $e_3=2d+1-(e_1+e_2)\leq d$. We
conclude that $\mu\neq 0$. Similarly, we conclude that
$\mu\neq \infty$. This finishes the proof of (3).
\end{proof}

The following example shows that the reduction of $f$ may be a
constant if we omit the assumption on the combinatorial type of $f$.

\begin{example}\label{exa:michelle}
We consider the rational function
\[
f(x)=\frac{px^4+x^2}{x^2+p}.
\]
The corresponding cover $f:\PP^1\to \PP^1$ is ramified at $6$ points,
each with ramification degree $2$, so it is not a Belyi map as considered
in this paper. Moreover, $x=1$ is not a ramification point.  
Both the
numerator and the denominator of $f$ have content $1$. Therefore with
our definition of the reduction we obtain
\[
\overline{f}=\frac{x^2}{x^2}=1.
\]
\end{example}

\begin{remark}\label{rem:silverman} In \cite[Section 2.3]{ADS}
Silverman gives a different definition of the reduction of $f$. The
difference between Silverman's definition and ours is (roughly
speaking) that he does not divide $f$ by the constant $c$ before
reducing, as we do in passing from $f$ to $\tilde{f}$. Instead,
Silverman only multiplies $f_1$ and $f_2$ by a common constant to
assume that at least one of the polynomials $f_1$ or $f_2$ has content
$1$.

We claim that in the case of a normalized Belyi map of ramification
type $(d; e_1, e_2, e_3)$, Silverman's definition agrees with ours. 
To see this, let $p$ be a prime. Recall that 
 \[ 
 1=f(1)=c\tilde{f}(1)=c \frac{\tilde{f}_1(1)}{\tilde{f}_2(1)}.
 \]
 Then Proposition \ref{prop:rednonconstant}.(3) implies that
 $\tilde{f}_1(1)$ and $\tilde{f}_2(1)$ have the same $p$-adic
 valuation, so $\tilde{f}(1)$ is a $p$-adic unit. Hence,
 $c=1/\tilde{f}(1)$ is a $p$-adic unit, as well. 

 Note that $c=c_1/c_2$, where $c_i$ is the content of the
polynomial $f_i$, so in particular $c$ is positive. We conclude that
$c\in \QQ_{>0}$ is a $p$-adic unit for all primes $p$, 
and hence that $c=1$. 
\end{remark}

Let $g\in \overline{\FF}_p(x)$ be a rational function. 
We say that the map $\PP^1\to\PP^1$ defined by $g$ is \textit{(in)separable} if $g$
is (in)separable. Recall that $g\in\overline{\FF}_p(x) $ is
inseparable if and only if it is contained in $\overline{\FF}_p(x^p)$.

\begin{defn}\label{def:goodred}
Let $f:\PP^1\to \PP^1$ have combinatorial type $(d; e_1, e_2,
e_3)$. Let $p$ be a prime. We say that $f$ has \textit{good reduction} if
the reduction $\overline{f}$ also has degree $d$.  If $\overline{f}$ is
additionally (in)separable, we say that $f$ has \textit{good (in)separable
reduction}. If $f$ does not have good reduction, we say it has \textit{bad reduction}. \end{defn}

In Corollary \ref{cor:coeff} we show that if $f$ has bad reduction, then $\overline{f}$ is inseparable. In particular, we do not have to consider the case of bad separable reduction.

Definition \ref{def:goodred} is the definition of good reduction used
in the theory of arithmetic dynamics. From the point of view of Galois
theory, one usually defines ``good reduction'' to mean good and
separable reduction. In our terminology  $f$ has bad reduction if and only if 
$\deg(\overline{f})<\deg(f)$.

\begin{prop}\label{prop:wild}
Let $f:\PP^1_\QQ\to \PP^1_\QQ$ be a normalized Belyi map of combinatorial
type $\underline{C}:=(d; e_1, e_2, e_3)$. 
Assume that the reduction $\overline{f}$ of $f$ to characteristic $p$
is separable. Then
\begin{itemize}
\item[(a)]
$f$ has good reduction (i.e., $\overline{d} = d$), and
\item[(b)]  $p \nmid e_i$ for all $i$.
\end{itemize}
\end{prop}

\begin{proof}
Our definition
of the reduction of $f$, together with the assumption that $f$ is
normalized, implies that the points $x=0, 1, \infty$ on the source
$\PP^1_{\QQ}$ specialize to pairwise distinct points of
$\PP^1_{\FF_p}$ (by Proposition \ref{prop:rednonconstant}.(2,3)). In
particular, multiplying $\overline{f}$ by a constant (if necessary),
we may assume that $\overline{f}$ is also normalized.

We write $f = f_1/f_2$ and $d_1 = \deg(f_1), d_2 = \deg(f_2)$.  We
denote the degree of $\overline{f}_i$ by $\overline{d}_i$, and define
$\overline{d}=\deg(\overline{f})$.  The polynomials $\overline{f}_1$
and $\overline{f}_2$ are not necessarily relatively prime. Put
$g=\gcd(\overline{f}_1, \overline{f}_2)$ and $\delta=\deg(g)$.

Let $\overline{e}_i$ be the ramification indices of $\overline{f}$ at
$x=0,1, \infty$, respectively.  Our first goal is to compare these to
the ramification indices $e_i$ of $f$. We start by considering what
happens at $x=0$. For this we write
\[
\overline{f}_i=gh_i, \qquad i=1,2.
\]
Since $\gcd(h_1, h_2)=1$ it follows that
\[
\overline{e}_1=\ord_0(\overline{f})=\ord_0(h_1).
\]
The definition of the reduction implies that
\[
\ord_0(\overline{f}_1)=\ord_0(g)+\ord_0(h_1)\geq \ord_0(f_1)=e_1.
\]
For the right-most equality we have used that $\gcd(f_1,
f_2)=1$. 
 Defining
$\varepsilon_1:=\ord_0(g)$
we obtain
\begin{equation}\label{eq:rede1}
\overline{e}_1+\varepsilon_1\geq e_1.
\end{equation}

Interchanging the roles of $x=0$ and  $x=1$, we similarly obtain 
\begin{equation}\label{eq:rede2}
\overline{e}_2+\varepsilon_2\geq e_2,
\end{equation}
where
$\varepsilon_2:=\ord_1(g)$. Note
that interchanging the roles of $x=0$ and $x=1$ corresponds to conjugating
$\overline{f}$ by $\varphi(x)=1-x$. 
From the definitions it follows immediately that
\begin{equation}\label{eq:redeps}
\varepsilon_1+\varepsilon_2\leq \delta.
\end{equation}

The definition of the reduction of $f$ and our normalization implies that
\begin{equation}\label{eq:reddelta}
d=d_1\geq \overline{d}_1=\overline{d}+\delta.
\end{equation}
Finally, for the ramification index $\overline{e}_3$ of $\overline{f}$ at
$\infty$ we have
\begin{equation}\label{eq:rede3}
d-e_3=d_2\geq \overline{d}_2=\overline{d}_1-\overline{e}_3=\overline{d}+
\delta-\overline{e}_3.
\end{equation}

Since we assume that $\overline{f}$ is separable, the Riemann--Hurwitz
formula applied to $\overline{f}$, together with the inequalities
(\ref{eq:rede1}), (\ref{eq:rede2}), (\ref{eq:redeps}),
(\ref{eq:reddelta}), and (\ref{eq:rede3}), yields
\begin{equation} \label{eq:ineq}
\begin{split}
-2&\geq
-2\overline{d}+(\overline{e}_1-1)+(\overline{e}_2-1)+(\overline{e}_3-1)\\
 &=(\overline{e}_1+\varepsilon_1-1)+(\overline{e}_2+\varepsilon_2-1)+
(\overline{e}_3-\overline{d}-\delta-1)+(\delta-\varepsilon_1-\varepsilon_2)+(-\overline{d})
\\& \geq -2d+(e_1-1)+(e_2-1)+(e_3-1)=-2.
\end{split}
\end{equation}

It follows that both inequalities are equalities. The fact that the
last inequality is an equality implies that $\overline{d}=d$, and that
the inequalities (\ref{eq:rede1}), (\ref{eq:rede2}),
(\ref{eq:redeps}), (\ref{eq:reddelta}), and (\ref{eq:rede3}) are also equalities. This proves Statement (a).

The first inequality in \eqref{eq:ineq} is an equality
if and only if all ramification of $\overline{f}$ is tame. Hence we have
$p\nmid \overline{e}_i$ for all $i$.
The statement $\overline{d}=d$ implies that
$\varepsilon_1 = \varepsilon_2 = \delta = 0$.  Hence $\overline{e}_i =
e_i$ for all $i$. Statement (b) follows.
\end{proof}

The following  is an immediate consequence of Lemma~\ref{prop:wild}.
\begin{cor}\label{cor:coeff}
Let $f:\PP^1_\QQ\to \PP^1_\QQ$ be a normalized Belyi map of combinatorial
type $\underline{C}:=(d; e_1, e_2, e_3)$. Assume that $f$ has bad
reduction to characteristic $p$. Then the reduction $\overline{f}$ is
inseparable.
\end{cor}

\section{Good inseparable monomial reduction}\label{sec:mainthm}

In Section \ref{sec:Dynamics} we determine the dynamical behavior of
separable covers $f$ of degree $d$ (of a given combinatorial type),
whose reduction modulo $p$ satisfies $\overline{f}(x) = x^d$. 
Since $1$ is a ramification point of $\overline{f}$, it
follows that $\overline{f}$ is inseparable, and hence that $p\mid
d$. If this happens, we say that $f$ has \textit{good (inseparable)
monomial reduction} to characteristic $p$.  In
Theorem \ref{thm:monored} we prove necessary and sufficient conditions
for this  to occur.

\begin{defn}\label{def:S} 
\begin{enumerate}
\item A rational map $\psi$ of degree $d$ in characteristic $p$ can be written 
uniquely as $\psi = \psi' \circ \phi^n$, where $\phi$ is the $p$-Frobenius map and
 $\psi'$ is separable. Suppose that $\psi'$ is a normalized Belyi map of
 combinatorial type $(d';e'_1,e'_2,e'_3)$. Then we call the
 $\overline{e}_i = p^n e'_i$ for $i=1,2,3$ the \emph{generalized
 ramification indices} of $\psi$; we allow $d'$ and each of the $e_i'$ to be trivial.
\item Let $\underline{C} = (d;e_1,e_2,e_3)$ be a combinatorial type
  such that $e_1 + e_2 + e_3 = 2d+1$. Then we define
\[
S_{\underline{C},p} := \{ \psi: \PP^1_{\overline{\FF}_p} \to \PP^1_{\overline{\FF}_p} \mid \psi \text{  satisfies the
  following combinatorial conditions} \}
\]
\begin{enumerate}
\item  $\deg(\psi) := \overline{d} \leq d$, and
\item there exist $\varepsilon_1, \varepsilon_2, \delta \geq 0$ such that 
\[
\varepsilon_1 + \varepsilon_2 \leq \delta \leq d - \overline{d}
\]
and the generalized ramification indices $\overline{e}_i$ ($i=1,2,3$)
of $\psi$ satisfy
\[
\begin{split}
\overline{e}_1 &\geq e_1 - \varepsilon_1,\\
\overline{e}_2 &\geq e_2 - \varepsilon_2 ,\\
\overline{e}_3 &\geq e_3 - (d - \overline{d} - \delta).
\end{split}
\]
\end{enumerate}
\end{enumerate}
\end{defn}

The set $S_{\underline{C}, p}$ may be considered as a
characteristic-$p$ analog of the set of normalized Belyi maps of
combinatorial type $\underline{C}$. Lemma \ref{Sred} and Proposition
\ref{4.2} below imply that this set consists of one element. Moreover,
it follows that $\psi\in S_{\underline{C}, p}$ is the reduction (in
the sense of \S\ref{sec:BadRed}) of the (unique) normalized Belyi map
of type $\underline{C}$ in characteristic zero. In particular, it
follows that $\psi\in S_{\underline{C}, p}$ may be defined over
$\FF_p$.

\begin{lemma}\label{Sred}
Let $f \colon \mathbb{P}^1 \to \mathbb{P}^1$ be a normalized cover in
characteristic zero of combinatorial type $\underline{C}=(d;
e_1,e_2,e_3)$. Its reduction $\overline{f}$ modulo $p$ lies in
$S_{\underline{C},p}$. 
\end{lemma}

\begin{proof}
If $f$ has good reduction at $p$, choose $\varepsilon_1
=\varepsilon_2=\delta =0$ and the result is immediate.  If $f$ has bad
reduction, the result follows immediately from the proof of Proposition
\ref{prop:wild}, for $\delta$ and $\varepsilon_i$ ($i=1,2$) as in that
proof.
\end{proof}

The following proposition is a reformulation in our terminology of a
result of Osserman.

\begin{proposition}\cite[Theorem 4.2.(i)]{ossermanrational}\label{4.2}
For any combinatorial type $\underline{C}$ and prime number $p$, we have $\vert S_{\underline{C},p} \vert = 1$. 
 \end{proposition}

We sketch the idea of Osserman's approach in his proof of Proposition
\ref{4.2}. For details we refer to \cite{ossermanrational} and
\cite{ossermanlinear}.  
Osserman interprets a rational map
$f:\PP^1_K\to\PP^1_K$ (up to automorphisms of the image) of degree $d$
over a field $K$ as a linear series by associating with the rational
map $f=f_1/f_2$ the $2$-dimensional vector subspace $V:=\langle f_1,
f_2\rangle$ of the polynomials of degree less than or equal to $d$.
This linear series may be considered as a point on the Grasmannian
$G(1, d)$. 
The condition that the map has ramification index at least $e_i$ at
the point $P_i$ defines a Schubert cycle $\Sigma_{e_i-1}(P_i)$ on
$G(1, d)$.  Base points of the linear series correspond to
common zeros of $f_1$ and $f_2$.

Consider an arbitrary linear series in positive characteristic, which
we denote by $\langle \psi_1, \psi_2\rangle$.   The inequalities (a) and (b) in
Definition \ref{def:S} may be interpreted as conditions on the linear
series. Note that we do not require the polynomials
$\psi_i$ to be relatively prime.  The zeros of $g:=\gcd(\psi_1, \psi_2)$ are base points of the
linear series. (Compare to the proof of Proposition \ref{prop:wild},
where we denoted the orders of these zeros at $0,1$ by
$\varepsilon_1, \varepsilon_2$, respectively.)

Assume that $\langle \psi_1, \psi_2\rangle$ is a linear series
satisfying the inequalities (a) and (b) from Definition \ref{def:S}.  The
Riemann--Hurwitz formula, together with the condition that
$2d+1=e_1+e_2+e_3$, implies that the linear series $\langle \psi_1,
\psi_2\rangle$ does not have base points if $\psi=\psi_1/\psi_2$ is
separable. (This follows as in the proof of Proposition
\ref{prop:wild}.)  However, the linear series associated with the
reduction of a normalized Belyi map (as defined above) may have base
points. Moreover, the base-point divisor
$D:=\varepsilon_1[0]+\varepsilon_2[1]-\delta[\infty]$ need not be
unique.  (See Example \ref{exa:basepts} below for an example.)

Our Proposition \ref{prop:rigid} states that in characteristic zero the
intersection of the three Schubert cycles $\Sigma_{e_1-1}(0)\cap
\Sigma_{e_2-1}(1)\cap \Sigma_{e_3-1}(\infty)$ has dimension $0$ and
the intersection product is $1$. In other words, the intersection
consists of one point. Osserman gives an intersection-theoretic
argument to prove the same statement in arbitrary characteristic. More
precisely, he proves that the intersection product of the three
Schubert cycles in positive characteristic is scheme-theoretically a
point. The underlying point is the unique map $\psi\in
S_{\underline{C}, p}$.  

In this paper, we are mainly interested in the case of good
inseparable reduction. In this case we have
$\varepsilon_1=\varepsilon_2=\delta=0$, hence the base-point divisor
is unique in this situation.\\

We give an example of a combinatorial type $\underline{C}$ for which
the reduction of the normalized dynamical Belyi map of this type has
base points. Moreover, we will see in that in this case the linear
series satisfying the inequalities (a) and (b) is not unique, even though
the underlying rational function is.

\begin{example}\label{exa:basepts}
This example is taken from Osserman \cite[\S
  2]{ossermanrational} (two paragraphs above Proposition 2.1).  Let
$p$ be a prime and $\underline{C}=(d; e_1, e_2, e_3)$ be a type such
that $d>p$ and $e_i<p$ for $i=1,2,3$. Then
\[
x^p\in S_{\underline{C},p}
\]
as one may verify directly. 

We therefore have that $\overline{d}=p=\overline{e}_i$ for all $i$. The
inequalities (a) and (b) from Definition \ref{def:S} become
\[
\varepsilon_1\geq 0, \qquad \varepsilon_2\geq 0, \qquad \delta\leq
d-e_3, \qquad \varepsilon_1+\varepsilon_2\leq \delta.
\]
We conclude that for a given combinatorial type $\underline{C}$ and
prime $p$ the base-point divisor
$D=\varepsilon_1[0]+\varepsilon_2[1]-\delta[\infty]$
need not be unique.  The linear series corresponding to a solution
$(\varepsilon_1, \varepsilon_2, \delta)$ of the inequalities is
\[
\langle x^pg, g\rangle,\qquad \text{ with }g=x^{\varepsilon_1}(x-1)^{\varepsilon_2}.
\]

Dynamical Belyi maps as considered in this example do exist; see
\cite[Cor.~2.5]{ossermanlinear}. Here is a concrete instance.
Choose $p\geq 7$ and $d$ with $p<d<3(p-1)/2$ and $k$ such that
$d-p<k<(p-1)/2$. 
Let $\underline{C}=(d; d-k, 2k+1, d-k)$.  The normalized Belyi map of this combinatorial type is
given in Proposition \ref{prop:GenFams4}.  The expression for the
coefficients $a_i$ in that lemma shows both that $p \vert a_i$ for
$d-p+1 \leq i \leq d$, and that $a_i \equiv (-1)^k a_{d-p-i} \bmod{p}$ for $0 \leq i \leq d-p$ (these are non-zero modulo $p$).
From this it follows that
\[
\overline{f}=x^p.
\]
Moreover, it follows that 
\[
g=\gcd(\overline{f}_1, \overline{f}_2)=(-1)^{d-p}a_{d-p}x^{d-p}+\cdots + a_0\in \FF_p[x]
\]
has degree $d-p$. The roots of the polynomial $g$ correspond to base points of the linear series. 
\end{example}

\begin{theorem}\label{thm:monored}
Suppose that $f \colon \mathbb{P}^1 \to \mathbb{P}^1$ is a normalized dynamical Belyi map of combinatorial type
$\underline{C}=(d=p^n d'; e_1,e_2,e_3)$, where $p\nmid d'$.  Then the
reduction $\overline{f}$ of $f$ modulo $p$ satisfies $\overline{f}(x)
= x^d$ if and only if $e_2 \leq p^n$.
\end{theorem} 

\begin{proof}
In the good monomial reduction case, i.e. where $\overline{f}(x)=x^d$,
we have $\overline{d} = d$, and the generalized ramification indices
are $\overline{e}_1 = \overline{e}_3 = d$, and $\overline{e}_2 = p^n$.
Hence, $e_2 \leq p^n$ is a necessary condition for good inseparable
monomial reduction.

Conversely, let $f$ be of combinatorial type $\underline{C}=(d=p^n d',
e_1,e_2,e_3)$ as in the statement of the theorem, and assume that
$e_2\leq p^n$. We claim that the map $g(x)=x^d$ lies in
$S_{\underline{C},p}$.

As before, we write $\psi$ as the composition of the purely inseparable
map of degree $p^n$ and the separable map $\psi'(x)=x^{d'}$, and we write
$e_1', e_2', e_3'$ for the ramification indices of $\psi'$ at $x=0, 1,
\infty$, respectively. Clearly, $\deg(\psi) =: \overline{d} = d$
satisfies $\overline{d} \leq d$. Moreover, the ramification indices
$\overline{e}_i$ of $g$ satisfy
\[
\begin{split}
\overline{e}_1&:=p^n e_1'=d\geq e_1,\\
\overline{e}_2&:=p^n \geq e_2,\qquad \text{(by assumption)},\\
\overline{e}_3&:=p^n e_3'=d\geq e_3.
\end{split}
\]
Choosing $\varepsilon_1=\varepsilon_2=\delta=0$, we see that $\psi$
satisfies the combinatorial conditions in Definition~\ref{def:S}, so
indeed $\psi \in S_{\underline{C},p}$. By Lemma \ref{Sred} and
Proposition \ref{4.2}, we obtain that $\overline{f} = g$, i.e., that
$f$ has good inseparable monomial reduction modulo $p$.
\end{proof}

\begin{remark}\label{specialcase}
Theorem \ref{thm:monored} can be viewed as a special case
of \cite[Theorem 2.4]{ossermanlinear}, which proves the existence of a
(necessarily unique) (in)separable cover for any combinatorial type
$(d; e_1,e_2,e_3)$ by studying the combinatorial properties of the
$e_i$. One can prove variants of the statement of
Theorem \ref{thm:monored} by specifying different possibilities for
the reduction $\overline{f}$ of $f$. 
This would amount to formulating conditions on the $e_i$ for reduction types other than the good inseparable monomial one.
\end{remark}

\begin{example}
Consider the combinatorial type $\underline{C} = (d=15; e_1, e_2,
e_3=d=15)$. The equation for the associated cover is given
in \cite[Proposition 5.1.2]{eskin}, and can alternatively be
determined from Proposition~\ref{prop:GenFams}. Computing the
reduction of $f$ modulo the primes $p=2,3,5,$ and $7$ yields the
following table. 
We immediately see that the results of the table are in accordance with Theorem \ref{thm:monored}.

\begin{center} 
\begin{longtable}{| c | l | l |}
\hline
 $e_2$ & $\overline{f}(x)$  at  $p=2$ &\textbf {Reduction Type} \\
\hline
$2$ & $x^{14}$ & bad \\
$3$ & $x^{15}+x^{14}+x^{13}$ & good separable \\
$4$ & $x^{12}$ & bad \\
$5$ & $x^{15}+x^{12}+x^{11}$ & good separable \\
$6$ & $x^{14}+x^{12}+x^{10}$ & bad \\
$7$ & $x^{15}+x^{14}+x^{13}+x^{12}+x^{11}+x^{10}+x^{9}$ & good separable \\
$8$ & $x^{8}$ & bad \\
$9$ & $x^{15}+x^{8}+x^{7}$ & good separable \\
$10$ & $x^{14}+x^{8}+x^{6}$ & bad \\
$11$ & $x^{15}+x^{14}+x^{13}+x^{8}+x^{7}+x^{6}+x^{5}$ & good separable \\
$12$ & $x^{12}+x^{8}+x^{4}$ & bad \\
$13$ & $x^{15}+x^{12}+x^{11}+x^{8}+x^{7}+x^{4}+x^{3}$ & good separable \\
$14$ & $x^{14}+x^{12}+x^{10}+x^{8}+x^{6}+x^{4}+x^{2}$ & bad \\
\hline \hline
$e_2$ & $\overline{f}(x)$ at $p=3$ &\textbf{Reduction Type} \\
\hline
$ e_2 \leq p=3$ & $x^{15}$ & good inseparable \\
$p=3 < e_2 \leq 2p=6$ & $2x^{15}+2x^{12}$ & good inseparable \\
$2p=6 < e_2 \leq 3p=9$ & $x^9$ & bad \\
$3p=9 < e_2 \leq 4p=12$ & $2x^{15}+x^9+x^6$ & good inseparable \\
$4p=12 < e_2 < 5p=d=15$ & $x^{15}+x^{12}+x^9+2x^6+2x^3$ & good inseparable \\
\hline \hline

$e_2$ & $\overline{f}(x)$ at $p=5$ &\textbf{Reduction Type} \\
\hline
$ e_2 \leq p=5$ & $x^{15}$ & good inseparable \\
$p=5 < e_2 \leq 2p=10$ & $3x^{15}+3x^{10}$ & good inseparable \\
$2p=10 < e_2 < 3p=d=15$ & $x^{15}+2x^{10}+3x^5$ & good inseparable \\
\hline \hline

$e_2$ & $\overline{f}(x)$ at $p=7$ &\textbf{Reduction Type} \\
\hline
$e_2 \leq 7$ & $x^{14}$ & bad \\
$ e_2 = 8$ & $5x^{15}+x^{14}+2x^8$ & good separable \\
$8 < e_2 \leq 14$ & $6x^{14}+2x^7$ & bad \\
\hline
\end{longtable}
\end{center}
\end{example}

%%%%%%%%%%%%%%%%%%%%%%%%%%%%%%%%%%%%%%%%%%%%%%%%%%%%%%%%%%%%%%%%%%%%%%%%%%%%%%%%

%\ifx
%    \thepage\undefined\def\jobname{ABEGKM_Covers} \input{ABEGKM_Covers}
%\fi

%%%%%%%%%%%%%%%%%%%%%%%%%%%%%%%%%%%%%%%%%%%%%%%%%%%%%%%%%%%%%%%%%%%%%%%%%%%%%%%%
\section{Dynamics}\label{sec:Dynamics}

%%%%%%%%%%%%%%%%%%%%%%%%%%%%%%%%%%%%%%%%%%%%%%%%%%%%%%%%%%%%%%%%%%%%%%%%%%%%%%%%

Let $f: \mathbb{P}^1 \to \mathbb{P}^1$ be a rational map and let
$f^n$ denote the $n$th iterate of $f$. The \emph{(forward) orbit} of
a point $P$ under $f$ is the set $\mathcal{O}_f(P) = \{ f^n(P): n \geq
0 \}$. The \emph{backward orbit} of a point $P$ under $f$ is the set
$\bigcup_{n=1}^\infty \{Q \in \mathbb{P}^1: f^n(Q) = P\}$. We say a
point $P \in \mathbb{P}^1$ is \emph{periodic} if $f^n(P) = P$ for some
positive integer $n$. The smallest such $n$ is called the \emph{exact
period} of $P$. For a point $P$ of exact period $n$, we define
the \emph{multiplier} of $f$ at $P$ to be the $n$th derivative of $f$
evaluated at $P$, denoted by $\lambda_P(f)$. A point $P$
is \emph{preperiodic} if $f^n(P)=f^m(P)$ for some positive integers
$n \neq m$. If $P$ is preperiodic but not periodic, we say it
is \emph{strictly preperiodic}. Let $\text{PrePer}(f,\mathbb{Q})$
denote the set of all rational preperiodic points for $f$. Our goal is
to determine $\text{PrePer}(f,\mathbb{Q})$ for an interesting class of
Belyi maps.

\begin{theorem}\label{qdynamicsGeneral} Let $f$ be a normalized Belyi map of combinatorial type $(d; e_1, e_2, e_3)$, where $d$ satisfies at least one of the following conditions:
\begin{enumerate}
\item $p=2$ is a divisor of $d$ with valuation $\ell = \nu_2(d)$, 
\item $p=3$ is a divisor of $d$ with valuation $\ell=\nu_3(d)$,
\item $d=p^\ell$ for some prime $p$.
\end{enumerate}
Assume that $e_2 \leq p^\ell$. Then $\text{PrePer}(f, \mathbb{Q})$
consists entirely of all rational fixed points for $f$ and their
rational preimages.
\end{theorem}

Recall that the condition $e_2\leq p^\ell$ implies that $f$ has good
 monomial reduction modulo $p$  (Theorem~\ref{thm:monored}).  To prove
 Theorem~\ref{qdynamicsGeneral}, we will use the following well-known
 theorem. 

\begin{theorem}\cite[Theorem 2.21]{ADS} \label{mrpe}
Let $f: \mathbb{P}^1_K \to \mathbb{P}^1_K$ be a rational function of degree $d \geq 2$ defined over a local field $K$ with residue field $k$ of characteristic $p$. Assume that $f$ has good reduction and that $P \in \mathbb{P}^1(K)$ is a periodic point for $f$ of exact period $n$. Let $m$ denote the exact period of $\overline{P}$ under the reduced map $\overline{f}$, and let $r$ denote the order of the multiplier $\lambda_{\overline{f}}(\overline{P})$ in $k^*$. Then one of the following holds:
\[n=m\]
\[ n=mr\]
\[n=mrp^e, e \in \mathbb{Z}, e>0.\]
\end{theorem}

\begin{proof}[Proof of Theorem~\ref{qdynamicsGeneral}]
Let $p$ be a prime in one of the three cases of the statement. To apply
Theorem \ref{mrpe},  we consider $f$ as element of $\QQ_p(x)$.

First suppose that $d=p^\ell$. When we reduce $f$
modulo $p$, we get $\overline{f}(x) = x^d$. All points in
$\PP^1(\mathbb{F}_p)$ are fixed points for $\overline{f}$. Moreover,
they are all critical points because the derivative of $\overline{f}$
is identically zero, so the multiplier of any point in $\mathbb{F}_p$
is zero. In the language of Theorem~\ref{mrpe}, for any
$\alpha \in \mathbb{Q}$ that is periodic under $f$, we have $m=1$ and
$r = \infty$. Therefore, $n=1$, so any rational periodic point for $f$
must be a fixed point.
 
If $2 \mid d$, reduce $f$ modulo $2$ to get $\overline{f}(x) = x^d$. All points in
$\PP^1(\mathbb{F}_2)$ are fixed and critical, so Theorem~\ref{mrpe}
implies that any periodic point for $f$ in $\mathbb{Q}$ must also be
fixed. 

Now assume that $3 \mid d$. In the case that $d$ is even, the points
in $\PP^1(\mathbb{F}_3)$ are all fixed under the reduction $\overline{f}$
of $f$ modulo $3$.  In the case that $d$ is odd, the points
$0,1, \infty$ are fixed and $\overline{f}(-1)=1$. This implies that $-1$ is
strictly preperiodic. In either case, the only
periodic points for $\overline{f}$ are fixed and critical, so once again
Theorem~\ref{mrpe} implies that all rational periodic points for $f$
must also be fixed points.

In all cases, the only periodic rational points for $f$ are
fixed points. Thus, $\text{PrePer}(f, \mathbb{Q})$ consists solely of
rational fixed points and their rational preimages.
\end{proof}

\begin{remark}\label{rem:generalize}
Each of the three conditions on primes dividing $d$ in Theorem~\ref{qdynamicsGeneral} ensures that all periodic points for the reduced map $\overline{f}$ are fixed points. This is not always true for arbitrary $d$ and $p$. For example, if $d=35$ and we reduce modulo $5$, the resulting map $\overline{f}(x) = x^{35}$ on $\mathbb{F}_5$ contains a periodic cycle of length two: $\overline{f}(2) = 3$ and $\overline{f}(3) = 2$. If we instead reduce modulo $7$, we see that $\overline{f}$ on $\mathbb{F}_7$ also has a $2$-cycle: $\overline{f}(2) = 4$ and $\overline{f}(4) = 2$. Thus in this case, we cannot use Theorem~\ref{mrpe} to deduce a statement analogous to that of Theorem~\ref{qdynamicsGeneral} because it is possible that $f$ contains a rational periodic point of exact period $2$.
\end{remark}

The following proposition gives a slightly stronger statement than
Theorem \ref{qdynamicsGeneral} in the first case of that theorem.

\begin{proposition}\label{prop:dynamicsp=2}
Let $f$ be the unique normalized Belyi map of combinatorial type $(d;
d-k, k+1, d)$. Write $\nu:=\nu_2(d)$ for the $2$-adic valuation of
$d$. Assume that $k+1\leq 2^\nu$. 
Then the only fixed points of $f$ in $\PP^1(\QQ)$ are $x=0, 1, \infty$.
\end{proposition}

\begin{proof} Recall from Theorem~\ref{thm:monored} that the condition 
 $k+1\leq 2^\nu$ implies that $f$ has good monomial reduction to
 characteristic $2$. As in Remark \ref{rem:altGenFams}, we write
\[
f(x)=x^{d-k}\left(\sum_{i=0}^k c_i(x-1)^i\right), \qquad \text{ with }
c_i=(-1)^i\binom{d-k+i-1}{i}.
\]
In particular, $c_0=1$. One easily checks that
\[
h(x):=\frac{f(x)-x}{x(x-1)}=
\left(\sum_{i=0}^{d-k-2}x^i+x^{d-k-1}\sum_{i=0}^{k-1}c_{i+1}(x-1)^i\right).
\]

Since $f$ is branched at $3$ points, we have that $d-k\geq 2$. It follows that
\[
h(0)\equiv 1\pmod{2}, \qquad h(1)=d-k-1-\binom{d-k}{1}\equiv 1\pmod{2}.
\]
Therefore the reduction $\overline{h}(x)$ of $h(x)$ modulo $2$ does not
have any roots in $\FF_2$, and hence $h$ does not have any roots
in $\QQ$.  Here we have used that~$h$ has good reduction to
characteristic $2$, i.e.~$\deg(h)=\deg(\overline{h})$. This implies that
$h$ does not have any rational roots that specialize to $\infty$ when
reduced modulo $2$. 
\end{proof}

We will now look at one particular family of normalized Belyi maps and
use Theorem~\ref{qdynamicsGeneral} to determine
$\text{PrePer}(f, \mathbb{Q})$. Let $d \geq 3$ be the degree of
$f$. Consider the following family:
\begin{equation} \label{fam1}
f(x) = -(d-1)x^d + dx^{d-1}.
\end{equation}
Recall from Example \ref{exa:GenFams} that this is the unique
normalized Belyi map of combinatorial type $(d; d-1, 2, d)$.

\begin{prop}\label{fixedptlemma}
Let $f$ be defined as in Equation~\eqref{fam1}. Then:
\begin{enumerate}
\item \label{fixed01} The only fixed points for $f$ in $\mathbb{P}^1(\mathbb{Q})$ are $0, 1$ and $\infty$ (for all $d$) and $\frac12$ (for $d=3$).
\item The only additional rational points in the backward orbits of these fixed points are $\frac{d}{d-1}$ (for all $d$) and $-\frac12$ (for $d=3$).
\end{enumerate}
\end{prop}

\begin{proof} 
\begin{enumerate}
\item The fixed points of $f$ are the roots of
 $f(x)-x=-(d-1)x^d+dx^{d-1}-x$, which factors as follows:
\[
  f(x)-x = x(x-1)(-(d-1)x^{d-2} + x^{d-3} + x^{d-4} + \ldots +x+1).
\]
By the rational root theorem, any nonzero rational zero of the above
polynomial is of the form $\frac{1}{b}$, where $b$ divides $d-1$. 
If $\frac{1}{b}$ is a root of
$f(x)-x$, then $b$ satisfies:
\begin{equation}\label{fixedb} 
\frac{b^{d-1}-1}{b-1}=b^{d-2} + b^{d-3} + \ldots +b  +1= d.
\end{equation}

\noindent \textbf{ Claim}:  Equation \eqref{fixedb} does not have any integer
solutions for $d\geq 4$.

Statement (1) immediately follows from the claim.
 
By inspection, it follows that $b\notin\{ 0, \pm 1\}$, so we may
assume $|b| \geq 2$. Note that we must have $b \leq -2$ because if
$b>1$, the left hand side of Equation~\eqref{fixedb} is strictly
greater than $d$. Moreover, since $b$ is negative, $d$ must be even,
since the left hand side of Equation~\eqref{fixedb} is positive. Since
$d \geq 4$ and $b \leq -2$, we have:
\[ 
\sum_{i=0}^{d-2} b^i > b^{d-2}+b^{d-3} =(-b)^{d-3}(-b+1)\geq 3\cdot  2^{d-3} > d.
\]
The claim follows.

\item We have the following by direct calculation:
\[f^{-1}(0) = \left\{ 0, \frac{d}{d-1} \right\} \]

If $d=3, f^{-1}(1) = \{1, -\frac12\}$.  Otherwise, if $d>3$, an
argument similar to that in Part~\ref{fixed01} shows that
$f^{-1}(1) \bigcap \mathbb{Q} = \{1\}$: Suppose that $f(\frac1b) = 1$,
where $b \in \mathbb{Z}$. (By the rational root theorem, any such
rational preimage is of this form.) Then, $f(\frac1b) -1=0$, which,
after factoring $(x-1)$ from the left hand side, gives the following
equation:
\[ \sum_{i=0}^{d-1} b^i = d.\]
  Note that $b=1$ is one
solution to this equation. Any other solution for $b$ would require $b
<0$ and in particular, $b \leq -2$. Therefore, $d$ must be odd for the
sum to be positive. If $d \geq 5$, we have the following:
\[ \sum_{i=0}^{d-1} b^i \geq b^{d-1}+b^{d-2} \geq |b|^{d-2} \geq 2^{d-2} > d.\]
Thus,  $f^{-1}(1) \bigcap \mathbb{Q} = \{1\}$.

A direct calculation also shows that if $d=3$, then $-\frac12$ has no rational preimages, and $\frac12$ has no rational preimages except itself. It remains to show that $\frac{d}{d-1}$ has no rational preimages. Suppose $f(\frac{a}{b}) = \frac{d}{d-1}$ for some relatively prime integers $a$ and $b$. The rational root theorem implies that $a|d$ and $b|(d-1)^2$. After clearing denominators, we have the following equation:
\begin{equation}\label{secondpreimageof0}
 -(d-1)^2a^d+d(d-1)a^{d-1}b-db^d = 0.  
\end{equation} 
Reducing modulo
 $d-1$ yields $-b^d \equiv 0$, so $(d-1)|b^d$. Reducing modulo $d$
 yields $-a^d \equiv 0$, so $d|a^d$. Let $p$ be a prime dividing
 $d$. Suppose that the valuation $\nu_p(d) = k \geq 1$ and $\nu_p(a)
 = \ell$, for $1 \leq \ell \leq k$. Then
 $\nu_p(-(d-1)^2a^d+d(d-1)a^{d-1}b-db^d) = k$ because
 $\nu_p(db^d) = k$ and
 $\nu_p(-(d-1)^2a^d+d(d-1)a^{d-1}b) \geq \text{max}\{\ell^d,
 k+\ell^{d-1}\}> k$. This contradicts
 Equation~\eqref{secondpreimageof0}. 
\end{enumerate}
\end{proof}

\begin{cor}\label{qdynamicsfam1}
Let $f$ be the polynomial of degree $d$ in the family defined in
Equation~\eqref{fam1}, where either $2 \mid d, 3 \mid d$, or $d=p^\ell$ for
some prime $p$. 
 Then:
\begin{enumerate}
\item $\text{PrePer}(f, \mathbb{Q}) = \{0, 1, \frac32, \frac12, -\frac12 , \infty\}$ if $d=3$.
\item $\text{PrePer}(f, \mathbb{Q}) = \left\{ 0, 1, \frac{d}{d-1}, \infty \right\}$ if $d \neq 3$. %\irene{Include $\infty$}
\end{enumerate}
\end{cor}

\begin{proof}
Theorem~\ref{qdynamicsGeneral} states that
$\text{PrePer}(f, \mathbb{Q})$ consists solely of fixed points for $f$
and their rational preimages.  Proposition~\ref{fixedptlemma} then
completely describes all rational preperiodic points for $f$.
\end{proof}

\begin{remark}\label{rem:dynamics} The statement of Proposition 
\ref{fixedptlemma}.(2) may be partially generalized.
For simplicity we restrict to the case that $f$ is the unique
normalized Belyi map of combinatorial type $(d; d-k, k+1, d)$.  An
explicit formula for $f$ was determined in
Proposition \ref{prop:GenFams}. We use the terminology of that result.

In the proof of Proposition \ref{prop:GenFams} we showed that the
derivative of $f$ satisfies
\[
f'(x)=(-1)^kcx^{d-k-1}(x-1)^k, \qquad \text{ with }c>0.
\]
Distinguishing $4$ cases depending on whether $k$ and $d$ are even or
odd and considering the sign of $f'$ yields the following statement
for the real elements in the fibers $f^{-1}(0)$ and $f^{-1}(1)$.
\begin{enumerate}
\item Suppose that $d$ and $k$ are both even. Then 
$f^{-1}(0)\cap \RR=\{0\}$ and $f^{-1}(1)\cap \RR=\{1, \beta\}$ for some $\beta<0$.
\item Suppose that $d$ is odd and $k$ is even. Then $f^{-1}(0)\cap \RR=\{0\}$ and
 $f^{-1}(1)\cap \RR=\{1\}$.
\item Suppose that $d$ is even and $k$ is odd. Then 
$f^{-1}(0)\cap \RR=\{0, \gamma\}$ for some $\gamma>1$ and
 $f^{-1}(1)\cap \RR=\{1\}$.
\item Suppose that $d$ and $k$ are both odd. Then 
$f^{-1}(0)\cap \RR=\{0, \gamma\}$ for some $\gamma>1$ and
$f^{-1}(1)\cap \RR=\{1, \beta\}$ for some $\beta<0$.
\end{enumerate}

In particular, this determines the rational values in $f^{-1}(0)$ and
$f^{-1}(1)$ in the case that $d$ is odd and $k$ is even.
In the other cases, in principle it is possible to analyze when the real roots
$\beta, \gamma$ are rational, similarly to the proof of
Proposition \ref{fixedptlemma}. 
In Proposition \ref{prop:GenFams} we showed that the leading coefficient of $f(x)$ is
$ca_0=(-1)^k\binom{d-1}{k}.$ It follows that if $\beta<0$ is a
rational root of $f(x)-1$, then we have
\[
\beta=\frac{-1}{b}  \qquad\text{ with }b\in \NN \text{ such that }
b \mid \binom{d-1}{k}.
\]

Similarly, assume that $\gamma>1$ is a rational root of $f(x)$. We use
the expression $f(x)=x^{d-k}\sum_{i=0}^k c_i(x-1)^i$ from
Remark \ref{rem:altGenFams}. Since $c_0= 1$ 
and
$c_k=\pm \binom{d-1}{k}$, we 
find
\[
\gamma=1+\frac{1}{c} \qquad\text{with }c\in \NN \text{ such that }c\mid 
\binom{d-1}{k}.
\] 
\end{remark}

\bibliographystyle{amsplain}
\bibliography{ABEGKM_Covers}
\end{document}